\documentclass{amsart}

\usepackage{color}

\newtheorem{Theorem}{Theorem}[section] 

\newtheorem{Definition}{Definition}[section]
\newtheorem{Lemma}{Lemma}[section]
\newtheorem{Proposition}{Proposition}[section]
\newtheorem{Remark}{Remark}[section]
\newtheorem{Assumption}{Assumption}[section]

\numberwithin{equation}{section}

\newcommand{\eps}{\varepsilon}
\newcommand{\E}{\mathbb{E}}

\RequirePackage[normalem]{ulem} 
\RequirePackage{color}\definecolor{RED}{rgb}{1,0,0}\definecolor{BLUE}{rgb}{0,0,1} 

\begin{document}
\title[Stochastic Perron's Method]{Stochastic Perron's method for Hamilton-Jacobi-Bellman equations 
}

\author{ Erhan Bayraktar}
\address{University of Michigan, Department of Mathematics, 530 Church Street, Ann Arbor, MI 48109.} \email{erhan@umich.edu.}
 \thanks{The research of E. Bayraktar was supported in part by the National Science Foundation under grants DMS 0955463 and DMS 1118673.}

 \author{Mihai S\^{\i}rbu}
 \address{University of Texas at Austin,
    Department of Mathematics, 1 University Station C1200, Austin, TX,
    78712.}  
    \email{sirbu@math.utexas.edu.} 
    \thanks{The research of
    M. S\^{\i}rbu was supported in part by the National Science
    Foundation under Grant
 DMS 1211988.} \thanks{Any opinions, findings, and conclusions or recommendations expressed in this material are those of the authors and do not necessarily reflect the views of the National Science Foundation. The authors would like to thank the anonymous referees for their careful reading and suggestions that led to an improvement of the paper. Special thanks go to Gordan \v Zitkovi\'c for his valuable comments.}
    
\date{\today}

  

\begin{abstract}
We show that the value function of a stochastic control problem is the unique solution of the associated Hamilton-Jacobi-Bellman (HJB) equation, completely avoiding  the proof of the
  so-called dynamic programming principle (DPP).  Using Stochastic Perron's method we construct a super-solution lying below the value function and a sub-solution dominating it.
A comparison argument easily  closes the proof. The  program has the precise meaning of verification for viscosity-solutions, obtaining the DPP as a conclusion. It also immediately follows that the weak and strong formulations of the stochastic control problem have the same value. Using this method we also capture the possible face-lifting phenomenon in a straightforward manner.
\end{abstract}

\maketitle 
%
%
\section{Introduction}

Stochastic Perron's method was introduced in \cite{bs2012a} for linear problems, and adapted in \cite{bs2012b} to free-boundary problems associated to Dynkin games. In the present paper, we carry out  a similar line of ideas, but with significantly different technicalities, to the most important case of non-linear problems, namely that of  Hamilton-Jacobi-Bellman equations (HJB)  in stochastic control. The result presented here actually represents the original motivation to introduce the stochastic version of Perron's method.

The goal of the paper is

\begin{enumerate}

\item  to prove the general result stating that ``the value function of a control problem is the unique viscosity of the associated HJB",
\item {\it without} having to first go through the proof of the Dynamic Programming Principle (DPP) but  actually obtaining it as a by-product,
\item in an as much {\it  elementary} manner as possible.
\end{enumerate}
The motivation for such goal is described in detail in \cite{bs2012a} and \cite{bs2012b}. To summarize, the program described by (1) and (2) (and, implicitly (3)) amounts to a verification result for viscosity solutions of HJB's. 
Overall, we believe this to be a genuinely novel approach to stochastic control,   that provides a deeper understanding of the relation between controlled diffusions and (viscosity) solutions of HJB's.

In addition to being a new  method to treat a fundamental class of problems, we  believe the program carried out here has two notable features, which basically amount to achieving our goals above:
\begin{enumerate}
\item This is a direct/verification approach to dynamic programming ( similar to  \cite{swiech-1} or \cite{swiech-2}), in that it {\it first} finds/constructs  a solution to the HJB, {\it then} shows that such solution is the value function,  avoiding altogether the Dynamic Programming Principle. However, this is technically very different from   the verification approach in \cite{swiech-1} or \cite{swiech-2}, and  can be viewed  as a probabilistic counterpart to the classical approach.
\item  We believe it, indeed,  to be  more elementary  than  either going through the probabilistic proof of the DPP (which is often incomplete, as described in the recent paper \cite{claisse-talay-tan} where some important details are fixed)  and then having to prove comparison of viscosity solutions anyway   or through the analytical techniques in \cite{swiech-1} and \cite{swiech-2}. In particular, there is no need for us to use  ``conditional controls'' or canonical spaces, usually needed in the proof of the DPP. These arguments are          still needed even in the recent proof of Bouchard and Touzi of a weaker version of the DPP  \cite{wdpp}.  While measurable selection arguments  are  circumvented there through the weaker formulation, the Markov arguments mentioned above are still present. 
Our method consists  only in applying It\^o's formula along the smooth test functions for viscosity solutions, plus an elementary stopping argument. In addition, arguments of the same type as we use here (maybe even more complicated) have to be used {\it anyway} when one uses the weak DPP to prove that the value function is a viscosity solution. Also, we  avoid the technicalities related to approximation by convolution and the approximation of the state equation in \cite{swiech-1} and \cite{swiech-2}.\end{enumerate}
We present here a  fresh look at a classic problem, so some comments on the existing literature are needed.  We mention briefly only those works that are closest related, at the risk of omitting relevant but further ideas. We first start with some important work in stochastic control, which, in the same spirit as our paper,  avoids the DPP.

Since our result amounts to  verification without smoothness, it is  conceptually  closest to \cite{swiech-1} and \cite{swiech-2}. Using approximation by convolution of viscosity semi-solutions in the deterministic case  (\cite{swiech-1})  and then also approximating the state equation by non-degenerate diffusions in the stochastic framework of two-player games (\cite{swiech-2}), the author performs a verification argument arriving at similar conclusions (in different situations though).  The probability space needs to accommodate an additional Brownian motion in the stochastic case, and, as mentioned above, the technicalities are very different and more involved, compared to our approach. Overall, the two approaches have little, if anything, in common.

At a formal level, one of our main results, Theorem \ref{thm:mthm2}, looks very much like the main result in the seminal work of Fleming-Vermes  \cite{fleming-vermes} and \cite{fleming-vermes-2} (see also Remark \ref{rem:fleming-vermes}).
\footnote{We would like to thank Ioannis Karatzas and Mete Soner for pointing out the closely related work of Fleming and Vermes.} More precisely, while the authors in  \cite{fleming-vermes} and \cite{fleming-vermes-2} show that the value function is the infimum of classical super-solutions, we show that, the value function is below the infimum of stochastic super-solutions, which is a viscosity sub-solution. 
While appearing stronger than our Theorem \ref{thm:mthm2} (considered by itself, without the other main result Theorem \ref{thm:mthm1}), the notable result in Fleming-Vermes has two features:
\begin{enumerate}
\item It contains a sophisticated approximation/separation  argument  used on top  of re-stating the  optimization problem as an infinite dimensional convex program,
\item It still uses the very definition of the value function,
\item By itself, is  not enough to show the value function is a viscosity sub-solution. Even if one does not mind the complicated approximation arguments, our Theorem \ref{thm:mthm2} is still needed on top of the very strong results in \cite{fleming-vermes} and \cite{fleming-vermes-2} to get such a conclusion. Even combining Fleming-Vermes with the Perron's method in Ishii \cite{ishii} would not yield this: the infimum over viscosity super-solutions may go below the value function, unless we now a {\it a-priori}    that the value function is a viscosity sub-solution, and we also have a comparison result (needed even for the viscosity version of Perron in \cite{ishii}). 
A sub-approximation counterpart to the work of Fleming-Vermes could close the argument, but this would still have a very different flavor than our work, since it uses, once again, the representation of the value function.
Actually, the recent papers \cite{tz1, tz2, tz3} carry along these lines, for path-dependent HJB's.
\end{enumerate}
 If one attempts to only use the Perron's method in Ishii \cite{ishii} to construct viscosity solutions, the same obvious obstacle described in relation to Fleming-Vermes arises: without additional knowledge on the properties of  value function, it does not compare with the output of Perron's method.


It should be also mentioned how our result compares to   other existing results about verification for viscosity solutions of HJB's, namely \cite{xyz}. The result in \cite{xyz} starts from the fact that the value function is the unique viscosity solution, and, using this piece of information, synthesizes the optimal control (if one exists) in terms of the generalized derivatives of the viscosity solution/value function. Our result plays a role  {\it before} the synthesis described in \cite{xyz}, and proves exactly that the value function is the unique viscosity solution, {\it without} resorting to the use of DPP.   In other words, our work addresses a different question than the one addressed in  \cite{xyz} (but quite similar to \cite{swiech-1} and \cite{swiech-2}).

The rest of the paper is organized as follows: In Section~\ref{sec:set-up}, we present the basic setup of the stochastic control problem, introduce the related HJB and the terminal condition. Moreover, we state our assumptions on the Hamiltonian. In Section~\ref{sec:subsoln}, we consider the strong formulation of the stochastic control problem and introduce the class of stochastic sub-solutions via which we construct a lower bound on the value function which is a viscosity super-solution. In Section~\ref{sec:ssuper}, we introduce the weak formulation of the stochastic control problem and introduce the class of stochastic super-solutions using which we construct a viscosity sub-solution to the HJB equation. Finally, in Section~\ref{sec:verification}, we verify that both value functions, in the weak and the strong formulation, equal the unique viscosity solution using comparison.

\section{Setup}\label{sec:set-up}
Let $U$ be a closed subset of $\mathbb{R}^k$ (the control space) and 
$\mathcal{O}$ an open subset of $ \mathbb{R}^d$ (the state space).
Let $b:[0,T] \times \mathcal{O} \times U \to \mathbb{R}^d$ and $\sigma:[0,T]\times\mathcal{O} \times U \to \mathbb{M}_{d,d'}$ be two measurable functions. We consider the controlled diffusion
\begin{equation}\label{eq:SDE}
 dX_t=b(t,X_t,u_t)dt+\sigma (t, X_t,u_t)dW_t, \ \ X\in \mathcal{O}.\\
 \end{equation}
We assume that the state lives in the open domain $\mathcal{O}\subset \mathbb{R}^d$, to include the treatment of  utility maximization models for positive wealth, which is popular in mathematical finance.
 Given a measurable function $g:\mathcal{O}\rightarrow \mathbb{R}$, our goal is to maximize the expected payoff received at a fixed time-horizon $T>0$ using progressively measurable processes $u$ taking values in $U$. Informally, we want to study the optimization problem 
$$\sup _{u}\mathbb{E}[g(X^u_T)],\ \ X_0=x\in \mathcal{O}.$$
\begin{Remark}
We choose only to maximize terminal payoffs, just to keep the notation simpler. In the literature, this is known as the Mayer formulation of stochastic control problems. The Bolza problem, which contains a running payoff as well, can be treated in an identical manner, with some additional notation.
\end{Remark}

One associates the following Hamiltonian to this problem:
\[
H(t,x,p,M):=\sup_{u \in U}\left[b(x,u)\cdot p+\frac{1}{2}Tr(\sigma(x,u)\sigma(x,u)^T M)\right],\ \ 0\leq t\leq T, \; x\in \mathcal{O}.
\] 
We make the following assumption on the Hamiltonian:
\begin{Assumption}\label{as:bnd}
Let us denote the domain of $H$ by
\[
\text{dom}(H):=\{(t,x,p,M) \in [0,T]\times \mathcal{O}\times \mathbb{R}^d \times \mathcal{S}_d: H(t,x,p,M)<\infty\}.
\]
We will assume that $H$ is continuous in the interior of dom$(H)$. Moreover, we will assume that
there exists a continuous function $G:[0,T] \times \mathcal{O} \times \mathbb{R}^d \times \mathcal{M}_d \to \mathbb{R}$ such that
\begin{enumerate}
\item $H(t,x,p,M)<\infty \implies G(t,x,p,M) \geq 0$,
\item $G(t,x,p,M)>0 \implies H(t,x,p,M)<\infty$.
\end{enumerate}
\end{Assumption} 
 \begin{Remark}
Our assumption above on the Hamiltonian $H$ differs from that of \cite{MR2533355}, which assumes that the domain of $H$ is closed. This latter assumption is well-suited for analyzing super-replication problems with volatility uncertainty but excludes the utility maximization problems. For example, our assumption works out well  for utility maximization problems, where
$\mathcal{O}=(0,\infty)$ and $G(t,x,p,M)=-M$.  Of course, one may ask why not simply choose $G=e^{-H}$? This is because, in general, $H$ is not jointly continuous everywhere as an extended value function. For example, in the case of one-dimensional utility maximization, where $H(t,x,p,M)=-p /2M^2$ for $M<0$, one can see that the Hamiltonian is not continuous at $(p,M)=(0,0)$, even if we view it as extended-valued. If $H$ is continuous everywhere, as an extended-valued mapping, then we can, indeed, choose $G=e^{-H}$. However, this is usually not the case.
 \end{Remark}
 
 Using the Stochastic Perron's Method, our goal is to show that, when a comparison principle is satisfied, the value function is, immediately, the unique viscosity solution of
  \begin{equation}\label{eq:PDE}
 \min\{-v_t(t,x)-H(t,x,v_x(t,x),v_{xx}(t,x)), G(t,x,v_x(t,x), v_{xx}(t,x))\}=0,
 \end{equation}
for $(t,x) \in [0,T)\times \mathcal{O}$, with the terminal condition
\begin{equation}\label{eq:bndry}
\min[v(T,x)-g(x),G(T,x,v_x(T,x), v_{xx}(T,x))]=0, \quad \text{on} \quad \mathcal{O},
\end{equation}
{\bf without} having to prove the dynamic programming principle.
\begin{Remark} One may question why we do not impose any kind of boundary conditions on $\partial \mathcal{O}$. This is because, as we can see from the assumptions below, we choose $\mathcal{O}$ as a natural domain, so that the controlled state  process $X$  never makes it to the boundary.
\end{Remark}
  \section{Stochastic sub-solutions}\label{sec:subsoln}
 
 In this section we will consider the so-called ``strong formulation'' of the stochastic control problem. 

The main goal of the paper is to outline how the  Stochastic Perron's Method in 
\cite{bs2012a} and \cite{bs2012b} can be used for the more important problem of Hamilton-Jacobi-Bellman equations. Having such goal in mind, 
but wanting to keep the presentation simpler, we make {\it quite restrictive  assumptions,  without losing the very interesting case when a boundary layer is present}. However, the restrictive assumptions we make are actually present in the important examples we have in mind. Our analysis can be  carried through under weaker assumptions, but, as it is customary in stochastic control, this would have to be done on a case by case basis, adapting the method to the specific optimization problem.
This is particularly important as far as admissibility is concerned.

Let $(\Omega, \mathcal{F}, \mathbb{P})$ be a probability space supporting an $\mathbb{R}^{d'}$-valued Brownian motion. Given $T$ let $\mathbb{F}:=\{\mathcal{F}_t, \;0 \leq t \leq T\}$ denote the completion of the natural filtration of this Brownian motion. (Note that $\mathbb{F}$ satisfies the usual conditions.)

\begin{Assumption}[{\bf State Equation}]\label{ass:stsolcoef}
 For any $(t, u) \in [0,T] \times U$ and $x,y \in \mathbb{R}^d$ we have
\begin{equation} 
\begin{split}
|b(t,0,u)|+|\sigma(t,0,u)| & \leq C (1+|u|), 
\\ |b(t,x,u)-b(t,y,u)|+|\sigma(t,x,u)-\sigma(t,y,u)| &\leq L(|u|) |x-y|,
\end{split}
\end{equation}
for some constant $C$ and some non-decreasing function $L:[0,\infty)\rightarrow [0, \infty)$.
\end{Assumption}
In what follows, we will work with controls and solutions defined on stochastic intervals. It is well know that, for deterministic intervals, one can choose integrands which are progressively measurable, optional or predictable, as they are equal up to equivalent classes. We choose here to work predictable controls, which are both the most general (i.e. work even for jump-diffusions) and best suited to handle joint-measurability in $(t,\omega)$ that is required on stochastic intervals.

 Admissibility (i.e. bounds or  integrability) is another very important issue, and we choose here a very small class of admissible process, namely bounded controls, but the bound is not fixed a-priori (unless the control space $U$ is itself compact). This allows to capture the full behavior of the value function, i.e. face lifting phenomenon, but the optimal control may not be admissible, if such  a control exists. This choice of admissible controls is the same as in Section 6 of Krylov \cite{MR2723141}, for the case of unbounded controls.
\begin{Definition}\label{def:admissible} Let $0\leq \tau \leq \rho \leq T$ be stopping times. By $\mathcal{U}_{\tau,\rho}$ we  denote the collection of 
predictable processes $u:(\tau, \rho]\rightarrow U$, by which we mean that the joint map
$$(0, T]\times \Omega \ni (t,\omega) \rightarrow u_t(\omega)\times 1_{[\tau (\omega )< t\leq \rho (\omega))}$$ is predictable with respect to the filtration $\mathbb{F}$ and which are uniformly bounded, i.e. there exists a positive constant $ 0\leq B(u)<\infty$ such that 
\[
  \|u\|:=\sup _{ \tau (\omega)\leq t<\rho (\omega)}|u_t(\omega)| \leq B(u).
\]

\end{Definition}
 Our definition of admissible control is very restrictive, in order to be able to   deal simultaneously with a  reasonably large class of problems. Of course, with this definition one does not expect an admissible optimal control to exist. However, if particular problems are considered, the definition of admissibility can be changed to  a larger class that does contain the optimal control (if such exists). For example,
\begin{enumerate}
\item in the case of utility maximization,  controls should only be locally integrable, and admissibility is a state constraint, namely that the wealth is non-negative, 
\item in the case of classical quadratic-type energy minimization, controls should be square integrable.
\end{enumerate}
Our proofs work verbatim in these particular cases.
 
\begin{Remark}
Assumption \ref{ass:stsolcoef}  on the controlled SDE, together with the Definition \ref{def:admissible} ensures that
 there is always a unique strong (adapted to $\mathcal{F}_t$) solution to the controlled SDE up to an explosion time. The additional Assumption  \ref{ass:domain} below actually means that there is never an explosion (for bounded controls). This is always the case if $\mathcal{O}=\mathbb{R}^d$, or in the case of utility maximization, if the control is the proportion of stocks held.
\end{Remark}


\begin{Assumption}[{\bf Natural Domain}]\label{ass:domain}
For any stopping times $\tau \leq \rho $ and any initial condition $\xi \in \mathcal{F}_{\tau}$ satisfying 
$\mathbb{P}(\xi \in \mathcal{O})=1$,  if $u\in \mathcal{U}_{\tau,\rho}$, the  unique strong solution $X^{u;\tau , \xi}$ of the SDE
\begin{equation}
\label{state-eq-full}
\left \{
\begin{array}{ll}
 dX_t=b(t,X_t,u_t)dt+\sigma (t, X_t,u_t)dW_t,\ \tau \leq t\leq \rho,\\
X_{\tau}=\xi

\end{array}
\right.
\end{equation}
does not explode, i.e.
$\mathbb{P}(X^{u;\tau, \xi}_t\in \mathcal{O},\ \ \tau \leq t\leq \rho)=1.$
\end{Assumption}

We  denote $\mathcal{U}_{0,T}$ by $\mathcal{U}$. Then let us define the value function by
\[
V(t,x)=\sup_{u \in \mathcal{U}_{t,T}}\E[g(X^{u;t,x}_T)], \ \ 0\leq t<T,\ x\in \mathcal{O}.
\]

 The goal of this section is to construct a super-solution of the Hamilton-Jacobi-Bellman equation \eqref{eq:PDE} with the terminal condition \eqref{eq:bndry}
that is smaller than the value function $V$. In order to do that, we need some growth property to be imposed on the pay-off function $g$ and the potential solutions of the PDE.  In this direction, we make an additional assumption:
\begin{Assumption}[{\bf Growth in $x$}] \label{ass:growth}
There exists a continuous and strictly positive gauge function
$\psi :\mathcal{O}\rightarrow (0, \infty)$ such that
\begin{enumerate}
\item for any  $\tau \leq \rho $ and any initial condition $\xi \in \mathcal{F}_{\tau}$, $\mathbb{P}(\xi \in \mathcal{O})=1$,  which satisfies
$\mathbb{E}[\psi (\xi)]<\infty$, 
 if the control $u$ is admissible, i.e. $u\in \mathcal{U}_{\tau,\rho}$, then
$$\mathbb{E}\left[\sup _{\tau \leq t \leq \rho}\psi (X^u_t)\right]<\infty\,;$$
\item $|g(x)|\leq C \psi (x)$ for some $C$.

\end{enumerate}
\end{Assumption}
 The assumption above is tailor-made to deal simultaneously with quadratic problems ($\mathcal{O}=\mathbb{R}^d, \psi (x)=|x|^2 $ {\bf or $\psi (x)=1+|x|^2$}) and utility maximization ($\mathcal{O}=(0,\infty)$, $\psi (x)=x^p$  or $\psi (x)=1+x^p$, $ -\infty <p<1$, $p\not= 0$). However, the choice of $\psi$ does matter, especially in the comparison principle that  we need for the terminal condition (see Remark \ref{face-lift}).
 \begin{Definition}\label{defn:stsubsoln}
 The set of stochastic sub-solutions for the parabolic PDE \eqref{eq:PDE}, denoted by $\mathcal{V}^-$, is  the set of functions $v:[0,T]\times \mathcal{O}\rightarrow \mathbb{R}$ which have the following properties:
 \begin{enumerate}
 \item[(i)] They are continuous and satisfy the terminal condition $v(T,\cdot) \leq g(\cdot)$ together with the growth condition
\begin{equation}\label{eq:grthv}
|v(t,x)|\leq C(v) \psi (x), 0\leq t\leq T,\  x\in \mathcal{O},\ \ \textrm{for \ some\ }\  C(v)<\infty.
\end{equation}
 \item[(ii)]  There exists a bound $L(v)<\infty$, depending on $v$, such that for each stopping time $\tau$ and each $\xi \in \mathcal{F}_{\tau}$ such that
$\mathbb{P}(\xi \in \mathcal{O})=1\  and \ \E[\psi (\xi)]<\infty,$
there exists a control $u \in \mathcal{U}_{\tau,T}$ defined on $[\tau,T]$ adapted to $\mathbb{F}$,  satisfying the bound
  $\|u\|\leq L(v)$ and  such that for any $\mathbb{F}$-stopping time $\rho \in [\tau,T]$ we have that
 \begin{equation}\label{eq:subm}
 v(\tau,\xi) \leq \E\left[v\left(\rho, X^{u;\tau,\xi}_{\rho}\right)\bigg|\mathcal{F}_{\tau} \right] \ a.s.
 \end{equation}
 \end{enumerate}
 \end{Definition}
 We do not expect the value function to be a stochastic sub-solution, except in the situations when there exists and admissible optimal control. As already mentioned, this is rarely the case, with our very restrictive definition  of admissibility. However, this does not cause any problem in the course of completing the Stochastic Perron Method: while the value function is not a sub-solution itself, it can be approximated by sub-solutions. 
 \begin{Remark}
 We ask for the sub-martingale property to hold only in between the fixed stopping time $\tau$ and any later  $\rho \geq \tau$, which is actually less than the full sub-martingale property on the stochastic interval $[\tau,T]$. 
 \end{Remark}
 
 \begin{Assumption}\label{ass:V-}
 We assume that $\mathcal{V}^{-} \neq \emptyset$. 
 \end{Assumption}
 \begin{Remark}
 Assumption~\ref{ass:V-} is satisfied, for example, when $g$ is bounded from below.
 \end{Remark}

 A crucial property of the set of stochastic solutions is the following stability result:
 \begin{Proposition}\label{prop:clsdundsup}
 If $v^1$ and $v^2$ are two stochastic sub-solutions, then $v=v^1 \vee v^2$ is also a stochastic sub-solution.
 \end{Proposition}
 
 \begin{proof}
 We will only show that $v$ satisfies item (ii) of the definition of stochastic sub solution.
  We can choose the uniform bound corresponding to $v$ as
 $$L(v)=L(v^1)\vee L(v^2).$$

 Now, fix a stopping time $\tau$ and a random variable $\xi \in \mathcal{F}_{\tau}$ with $\mathbb{P}(\xi \in \mathcal{O})=1$ and $\mathbb{E}[\psi (\xi)]<\infty.$ Then, by the definition of the stochastic sub-solutions $v^1$ and $v^2$,  it follows that there are two controls $\|u_1\|\leq L(v^1)$ and $\|u_2|\leq L(v^2)$ satisfying
 \[
 v^i(\tau,\xi) \leq \E[v^i(\rho, X^{u_i;\tau,\xi}_{\rho})|\mathcal{F}_{\tau}], \quad i \in \{1,2\}.
 \]
Now define a control $u$ (on the stochastic interval $(\tau,T]$) by
\begin{equation}\label{eq:pstcntrl}
u=1_{\{v^{1}(\tau,\xi) \geq v^{2}(\tau,\xi)\}}u_1+1_{\{v^{1}(\tau,\xi) <v^{2}(\tau,\xi)\}}u_2.
\end{equation}
Now, for each $\tau \leq \rho \leq T$, we have
\begin{enumerate}
\item on  $\{v^{1}(\tau,\xi) \geq v^{2}(\tau,\xi)\} \in \mathcal{F}_{\tau}$ we have
\[
v^{1}(\rho, X^{u_1;\tau,\xi}_{\rho})=v^{1}(\rho, X^{u;\tau,\xi}_{\rho}) \leq v(\rho, X^{u;\tau,\xi}_{\rho});
\]
\item  on  $\{v^{1}(\tau,\xi)< v^{2}(\tau,\xi)\} \in \mathcal{F}_{\tau}$ we have 
\[
v^{2}(\rho, X^{u_2;\tau,\xi}_{\rho})=v^{2}(\rho, X^{u;\tau,\xi}_{\rho}) \leq v(\rho, X^{u;\tau,\xi}_{\rho}).
\]
\end{enumerate}
Applying the definition of sub-solutions for $v^1$ and $v^2$ (for controls $u_1$ and $u_2$) we get
\[
1_{\{v^{1}(\tau,\xi) \geq v^{2}(\tau,\xi)\}} v^{1}(\tau,\xi) \leq \E \left[1_{\{v^{1}(\tau,\xi) \geq v^{2}(\tau,\xi)\}}v^{1}(\rho, X^{u;\tau,\xi}_{\rho})\bigg|\mathcal{F}_{\tau}\right]  \ a.s.,
\]
since $\{v^{1}(\tau,\xi) \geq v^{2}(\tau,\xi)\}\in \mathcal{F}_{\tau}$. Therefore, according to item (1) above we have
\begin{equation}\label{eq:ooft}
1_{\{v^{1}(\tau,\xi) \geq v^{2}(\tau,\xi)\}} v^{1}(\tau,\xi) \leq \E \left[1_{\{v^{1}(\tau,\xi) \geq v^{2}(\tau,\xi)\}}v(\rho, X^{u;\tau,\xi}_{\rho})\bigg|\mathcal{F}_{\tau}\right]  \ a.s.
\end{equation}
Similarly, we obtain
\[
1_{\{v^{1}(\tau,\xi) < v^{2}(\tau,\xi)\}} v^{2}(\tau,\xi) \leq \E \left[1_{\{v^{1}(\tau,\xi) < v^{2}(\tau,\xi)\}}v^{2}(\rho, X^{u;\tau,\xi}_{\rho})\bigg|\mathcal{F}_{\tau}\right]  \ a.s.,
\]
and by item (2) above we have
\begin{equation}\label{eq:toft}
1_{\{v^{1}(\tau,\xi) < v^{2}(\tau,\xi)\}} v^{2}(\tau,\xi) \leq \E \left[1_{\{v^{1}(\tau,\xi) < v^{2}(\tau,\xi)\}}v(\rho, X^{u;\tau,\xi}_{\rho})\bigg|\mathcal{F}_{\tau}\right]  \ a.s.
\end{equation}
Putting \eqref{eq:ooft} and \eqref{eq:toft} together we conclude.
\end{proof}

\begin{Theorem}\label{thm:mthm1}
\emph{(The supremum of stochastic sub-solutions is a viscosity super-solution)}
Let Assumptions~\ref{as:bnd}-(1), \ref{ass:stsolcoef}, \ref{ass:domain}, \ref{ass:growth} and \ref{ass:V-} hold true. Assume also that $g$ is a lower semi-continuous function and $V<\infty$.
Then the lower stochastic  envelope
 \[
 v^{-}:=\sup_{v \in \mathcal{V}^-}v \leq V<\infty, 
 \]
  is a viscosity super-solution of \eqref{eq:PDE}. 
 Moreover, if we define 
 \begin{equation}\label{eq:cfdfn}
 v^-(T-,x):=\liminf_{(t'<T,x') \to (T, x)}v^-(t',x'), \quad x \in \mathcal{O},
 \end{equation}
 then
the function $v^-(T-,\cdot) (\geq g(\cdot))$ is a viscosity super-solution of \eqref{eq:bndry}. 
 \end{Theorem}
\begin{Remark}
The function $v^-$ may not have a limit from the left at $t=T$.
We, therefore, modify this function as described in \eqref{eq:cfdfn} at $t=T$. If we consider the function $v^-$ with the new terminal condition $v^{-}(T-,\cdot)$, it still is lower-semi continuous, as it is used in the proof of Theorem \ref{thm:main-comp}.
\end{Remark}

\begin{proof} \hfill
\\ \textbf{Step 1.} The fact that $v^{-} \leq V$ follows directly from the definition of the class of stochastic sub-solutions and by the definition of $\mathcal{U}_{t,T}$ and $V$. \\
 \textbf{Step 2.} \emph{The interior super-solution property.}
Let $\varphi:[0,T] \times \mathcal{O} \to \mathbb{R}^d$ be a $C^{1,2}$-test function such that $v^{-}-\varphi$ attains a strict local minimum equal to zero at some  parabolic interior point $(t_0,x_0) \in [0,T) \times \mathcal{O}$.
We first  prove that 
 \begin{equation}\label{eq:supsolpde}
 -\varphi_t(t_0,x_0)-H(t_0,x_0,\varphi_x(t_0,x_0),\varphi_{xx}(t_0,x_0)) \geq 0,
 \end{equation}
  by contradiction. To this end, assume that
\[
(-\varphi_t-\sup_{u}L^{u}_t\varphi )(t_0,x_0)<0.
\]
But then there exists $\tilde{u} \in U$ such that 
\begin{equation}\label{eq:tildeu}
(-\varphi_t-L^{\tilde{u}}_t\varphi )(t_0,x_0)<0.
\end{equation}

Since the coefficients of the SDE are continuous there exists a small enough ball $B(t_0,x_0,\eps)$ such that
\[
-\varphi_t-L^{\tilde{u}}_t\varphi (t,x)<0, \quad (t,x) \in B(t_0,x_0,\eps),
\]
and 
\[
\varphi(t,x)<v^{-}(t,x), \quad  (t,x) \in B(t_0,x_0,\eps)-\{(t_0,x_0)\}.
\] 
To be precise, all along the paper, we use the norm $\|(t,x)\|=\max\{|t|, |x|\}$, so 
$$B(t_0,x_0,\eps):=\{(t,x)\in [0,T)\times \mathcal{O} | \max\{|t-t_0|, |x-x_0|\}<\eps\}.$$
Since $v^{-}-\varphi$ is lower semi-continuous and $\overline{B(t_0,x_0,\eps)}-B(t_0,x_0,\eps/2)$ is compact, there exists a $\delta>0$ satisfying
\[
v^{-}-\delta \geq \varphi \quad \text{on}\;\overline{B(t_0,x_0,\eps)}-B(t_0,x_0,\eps/2).
\]
Using Proposition 4.1 in \cite{bs2012a} together with Proposition \ref{prop:clsdundsup} above, we obtain a (countable) increasing sequence of stochastic sub-solutions
$v_n\nearrow v^-$. Now, since $\varphi$ is continuous, as well as $v_n$'s, we can use a Dini argument (identical to the one  in Lemma 2.4 of \cite{bs2012b})  to conclude  that for $\delta' \in (0,\delta)$ there exists a stochastic sub-solution $v=v_n$ (for some large enough $n$) such that
\begin{equation}
v-\delta' \geq \varphi \quad \text{on}\;\overline{B(t_0,x_0,\eps)}-B(t_0,x_0,\eps/2).
\end{equation}
Choosing $\eta \in (0,\delta')$ small enough we have that the function
\[
\varphi^{\eta}:=\varphi+\eta
\]
satisfies
\[
\begin{split}
-\varphi^{\eta}_t-L^{\tilde{u}}_t\varphi^{\eta}(t,x)&<0, \quad (t,x) \in B(t_0,x_0,\eps),
\\ \varphi^{\eta}(t,x)&<v(t,x), \quad (t,x) \in \overline{B(t_0,x_0,\eps)}-B(t_0,x_0,\eps/2),
\end{split}
\]
and
\[
\varphi^{\eta}(t_0,x_0)=v^{-}(t_0,x_0)+\eta >v^-(t_0,x_0).
\]
Now we define 
$$v^{\eta}=
\left \{
\begin{array}{l}
 v\vee  \varphi ^{\eta} \textrm{~on~} \overline{ B(t_0, x_0, \varepsilon)},\\
v \textrm{~outside~}\overline{ B(t_0, x_0, \varepsilon)}.
\end{array}
\right.
$$
Clearly, $v^{\eta}$ is continuous and $v^{\eta}(t_0,x_0)=\varphi^{\eta}(t_0,x_0)>v^{-}(t_0,x_0)$. And since $\eps$ can be chosen so that $T>t_0+\eps$, $v^{\eta}$ satisfies the terminal condition.  In addition, the growth condition in (i)  Definition~\ref{defn:stsubsoln} holds for $v^{\eta}$, since such growth condition holds for the approximate supremum $v$ (although we may not have, without additional assumptions, a similar growth condition on $v^-$).

We only need to show that $v^{\eta}$ satisfies (ii) in Definition~\ref{defn:stsubsoln}  to get a contradiction and complete the proof. Let $0\leq \tau \leq T$ be a fixed stopping time and 
$\xi \in \mathcal{F}_{\tau}$, $\mathbb{P}(\xi \in \mathcal{O})=1,$ such that
$\mathbb{E}[\psi (\xi)]<\infty.$ We need to construct a control $u\in \mathcal{U}_{\tau,T}$ that works for $v^{\eta}$ in (ii) in Definition~\ref{defn:stsubsoln}. Following the arguments in the proof of Proposition~\ref{prop:clsdundsup}, such  a control $u$ can be constructed in a surprisingly simple way, {\it which represents a significant technical improvement over the previous work \cite{bs2012a} or \cite{bs2012b}}.

Denote by $u_0\in \mathcal{U}_{\tau ,T}$ the control corresponding to initial time $\tau $ and initial condition $\xi$ in (ii) in Definition~\ref{defn:stsubsoln} for the stochastic sub-solution $v$.  Denote by $A$ the event
$$A=\{(\tau, \xi)\in B(t_0, x_0, \varepsilon/2) \textrm{~and~} \varphi ^{\eta}(\tau ,\xi)>v(\tau ,\xi)\}.$$
Recalling \eqref{eq:tildeu}, we 
define the new  admissible control $u_1\in \mathcal{U}_{\tau ,T}$ by
$$u_1=\tilde{u}\times 1_A
+u_0 \times 1_{A^c},$$
and by $\tau _1$ the first time after $\tau$ when the diffusion started at $\xi$ and controlled by $u_1$ hits the boundary of $B(t_0, x_0, \varepsilon/2)$:
$$\tau _1=\inf \{\tau\leq  t\leq T| X_t^{u_1;\tau, \xi}\in \partial B(t_0, x_0, \varepsilon/2)\}.$$
Now, denote by
$$\xi _1=  X_{\tau _1}^{u_1;\tau, \xi}\in \partial B(t_0, x_0, \varepsilon/2),$$
and by $u_2\in \mathcal{U}_{\tau _1,T}$ the control in (ii) in Definition~\ref{defn:stsubsoln} corresponding to $v$ for the starting time $\tau _1$ and initial condition $\xi_1$. Now, we can finally define
$$u=u_1\times 1_{\{\tau <t\leq \tau _1\}}+u_2 \times 1_{\{\tau_1 < t \leq T\}}.$$
Note that the control $u$ is bounded by $L(v) \vee |\tilde{u}|$, and, therefore, it is admissible. Consider any stopping time $\rho$ such that $\tau \leq \rho\leq T$.  On the event $A$, $\varphi ^{\eta}(\cdot, X_{\cdot})$ is a sub-martingale up to $\rho \wedge \tau _1$ (because of It\^o's formula together with the fact that $\varphi ^{\eta}$ is bounded in the interior ball), which reads
$$1_A \varphi ^{\eta}(\tau, \xi)\leq \mathbb{E}[1_A \varphi ^{\eta}(\rho \wedge \tau _1, X_{\rho \wedge \tau _1}^{\tilde{u};\tau,\xi })|\mathcal{F}_{\tau}] \ a.s.$$
Since 
$$1_A \varphi ^{\eta}(\rho \wedge \tau _1, X_{\rho \wedge \tau _1}^{\tilde{u};\tau,\xi})=1_A \varphi ^{\eta}(\rho \wedge \tau _1, X_{\rho \wedge \tau _1}^{u;\tau,\xi})\leq 1_A v ^{\eta}(\rho \wedge \tau _1, X_{\rho \wedge \tau _1}^{u;\tau;\xi}),$$
we actually obtain 
$$1_A v^{\eta}(\tau, \xi)=1_A \varphi ^{\eta}(\tau, \xi)\leq \mathbb{E}[1_A v ^{\eta}(\rho \wedge \tau _1, X_{\rho \wedge \tau _1}^{u;\tau,\xi})|\mathcal{F}_{\tau}]  \ a.s.$$
Next, we use the fact that $u_1$ is the ``optimal'' control for $v$,  together with $v=v^{\eta}$ everywhere outside the open ball $ B(t_0,x_0,\eps/2)$, to obtain:
$$1_{A^c} v^{\eta}(\tau, \xi)=1_{A^c} v (\tau, \xi)\leq \mathbb{E}[1_{A^c} v (\rho \wedge \tau _1, X_{\rho \wedge \tau _1}^{u_1;\tau.\xi})|\mathcal{F}_{\tau}]= \mathbb{E}[1_{A^c} v ^{\eta}(\rho \wedge \tau _1, X_{\rho \wedge \tau _1}^{u;\tau,\xi})|\mathcal{F}_{\tau}].$$
Putting the above together, we obtain:
 \begin{equation}\label{eq:submp}
 v^{\eta}(\tau,\xi) \leq \E\left[v^{\eta}\left(\rho \wedge \tau_1, X^{u;\tau,\xi}_{\rho \wedge \tau_1}\right)\bigg|\mathcal{F}_{\tau} \right]\ a.s.
 \end{equation}
Let us introduce yet another notation: $B=\{\rho \leq \tau _1\}\in \mathcal{F}_{\tau _1}.$ We know that, on the boundary $ \partial B(t_0, x_0, \varepsilon/2)$, $v=v^\eta.$ Applying the definition of $u$, together with the fact that $u_2$ is ``optimal'' for $v$ starting at $\tau_1$ with condition $\xi_1$, we have
$$1_{B^c} v^{\eta}(\tau _1, \xi_1)=1_{B^c} v (\tau _1, \xi_1)\leq \mathbb{E}[1_{B^c} v (\rho , X_{\rho}^{u_2;\tau _1,\xi _1})|\mathcal{F}_{\tau _1}]\leq \mathbb{E}[1_{B^c} v ^{\eta}(\rho, X_{\rho}^{u;\tau,\xi})|\mathcal{F}_{\tau _1}].$$
If we rewrite the RHS in 
\eqref{eq:submp} as
$$\E\left[v^{\eta}\left(\rho \wedge \tau_1, X^{u;\tau,\xi}_{\rho \wedge \tau_1}\right)\bigg|\mathcal{F}_{\tau} \right]=\E\left[ 1_B v^{\eta}\left(\rho, X^{u;\tau,\xi}_{\rho}\right) 
+ 1_{B^c} v^{\eta}(\tau _1, \xi_1)\bigg|\mathcal{F}_{\tau} \right],
$$
and use the tower property, we get, indeed
 \begin{equation*}
 v^{\eta}(\tau,\xi) \leq \E\left[v^{\eta}\left(\rho, X^{u;\tau,\xi}_{\rho}\right)\bigg|\mathcal{F}_{\tau} \right] \ a.s.
 \end{equation*}
This completes the proof of \eqref{eq:supsolpde}, from which it follows that
$$H(t_0,x_0,\varphi_x(t_0,x_0),\varphi_{xx}(t_0,x_0))<\infty.$$
Thanks to Assumption~\ref{as:bnd}-(1) we also have that
\begin{equation}\label{eq:usdtbstpth}
G(t_0,x_0,\varphi_x(t_0,x_0),\varphi_{xx}(t_0,x_0)) \geq 0,
\end{equation}
finishing the proof of interior super-solution property.

 \textbf{Step 3.} \emph{The terminal condition, Part I.} We will show that $v^{-}(T,\cdot)=g(\cdot)$.
Assume that for some $x_0 \in \mathcal{O}$ we have 
\[
v^-(T,x_0)<g(x_0).
\]
We will use this information to construct a contradiction. Since $g(\cdot)$ is lower-semi continuous then there exists an $\eps>0$ such that
\[
g(x)\geq v^-(T,x_0)+\eps, \quad \text{if}\quad |x-x_0| \leq \eps.
\]
Due to the fact that $v^-$ is lower-semi continuous, it is bounded from below on the compact set
\[
(\overline{B(T,x_0,\eps)}-B(T,x_0,\eps/2))\cap ([0,T] \times \mathcal{O}).
\]
For a small enough $\eta>0$ we have that 
\[
v^{-}(T,x_0)-\frac{\eps^2}{4 \eta}<-\eps+\inf_{(t,x) \in (\overline{B(T,x_0,\eps)}-B(T,x_0,\eps/2))\cap ([0,T] \times \mathcal{O})}v^{-}(t,x).
\]
Since the above inequality is strict, following the proof of Step 2 in Theorem \ref{thm:mthm1}, we use again Proposition 4.1 in \cite{bs2012a} together with Proposition \ref{prop:clsdundsup} above, and  a  Dini argument to find  a stochastic sub-solution $v \in \mathcal{V}^-$ such that
\begin{equation}\label{eq:consl24}
v^{-}(T,x_0)-\frac{\eps^2}{4 \eta}<-\eps+\inf_{(t,x) \in (\overline{B(T,x_0,\eps)}-B(T,x_0,\eps/2))\cap ([0,T] \times \mathcal{O})}v(t,x).
\end{equation}
For $k>0$ define
\[
\varphi^{\eta,\eps,k}(t,x)=v^{-}(T,x_0)-\frac{|x-x_0|^2}{\eta}-k(T-t).
\]
Choose $k$ large enough, at least as large as $k\geq \varepsilon/4\eta$ but possibly much larger,  such  that 
\[
\left[-\varphi^{\eta,\eps,k}_t-\sup_{u}L^{u}_t\varphi^{\eta,\eps,k}\right](t_0,x_0)<0 \quad \text{on} \quad \overline{B(T,x_0,\eps)}.
\]
Using \eqref{eq:consl24} we obtain
\[
\varphi^{\eta,\eps,k} \leq -\eps+v \quad \text{on} \quad  (\overline{B(T,x_0,\eps)}-B(T,x_0,\eps/2))\cap ([0,T] \times \mathcal{O}).
\]
On the other hand,
\[
\varphi^{\eta,\eps,k} (T,x) \leq v^{-}(T,x_0) \leq g(x)-\eps, \quad \text{for} \quad |x-x_0| \leq \eps.
\]
Now, let $\delta<\eps$ and define
\[
v^{\eps,\eta,k,\delta}:=\left\{ 
\begin{array}{l} v \vee (\varphi^{\eps,\eta,k}+\delta) \quad \text{on} \quad \overline{B(T,x_0,\eps)}, \\
v \quad \text{outside} \quad  \overline{B(T,x_0,\eps)} .
\end{array}\right.
\]
Now using the idea in Step 1 of the proof, we can show that $v^{\eps,\eta,k,\delta} \in \mathcal{V}^-$ but $v^{\eps,\eta,k,\delta}(T,x_0)=v^{-}(T,x_0)+\delta>v^{-}(T,x_0)$, leading to a contradiction. 

The only reason we actually proved $v^-(T, \cdot)= g(\cdot)$ was to get some information about the left liminf $v^-(T-, \cdot)$.  More precisely, since $v^-$ is lower semi-continuous, we know that
$$g(\cdot)= v^-(T, \cdot)\leq v^-(T-, \cdot).$$
In  order to finish the proof of the Theorem, we only need to show that  $v^-(T-,\cdot)$ is a viscosity super-solution of \eqref{eq:bndry}, which we will do in the next step.

\noindent \textbf{Step 4.} \emph{The terminal condition, Part II.} We show  that the l.s.c. function  $v^{-}(T-,\cdot)$ is a viscosity super-solution of 
 \[
 G(T,x,v^{-}_x(T,x), v^{-}_{xx}(T,x)) \geq 0, \quad x \in \mathcal{O}.
 \]
 
 
 The  arguments used below trace back to \cite{bcs} and were  technically refined later for more general models of super-hedging  in \cite{cpt},  \cite{st}  and others, as presented in the  survey paper \cite{st-lnm}.  We basically use the notation from  Lemma 4.3.2 in \cite{MR2533355} which summarizes the existing literature.

 More precisely, we  rely on the fact that $v^{-}$ satisfies the same equation in the interior, a fact we established in Step 2, to get information about the limit  as $t\rightarrow T$.
Let $y \in \mathbb{R}^{d}$ and $\psi(x)$ be a test function satisfying
\begin{equation}\label{eq:tfspsl}
0=v^{-}(T-,y)-\psi(y)=\min_{x \in \mathbb{R}^d}(v^{-}(T-,x)-\psi(x)).
\end{equation}
By the very definition  of $v^{-}(T-, \cdot)$, there exists a sequence $(s_m,y_m)$ converging to $(T,y)$ with $s_m<T$ such that
\[
\lim_{m \to \infty} v^{-}(s_m,y_m)=v^{-}(T-,y).
\]
Let us construct another test function that depends both on $t$ and $x$ variables:
\[
\psi_m(t,x)=\psi(x)-|x-y|^4+\frac{T-t}{(T-s_m)^2},
\]
and choose $(t_m,x_m) \in [s_m,T] \times \overline{B(y,\varepsilon)}$ as a minimum of $v^{-}-\psi_m$ on $[s_m,T] \times \overline{B(y,\varepsilon)}$ where $\varepsilon$ is chosen small enough so that $\overline{B(y,\varepsilon)}\subset \mathcal{O}$.

What we would like to do next is to show that in fact $t_m<T$ for large enough $m$ and that $x_m \rightarrow y$. The first fact follows from the observation that
\[
v^-(s_m,y_m)-\psi_m(s_m,y_m) \leq -\frac{1}{2(T-s_m)}<0,
\]
for large enough $m$ and that
\[
v^{-}(T-,x)-\psi_m(T,x)\geq v^{-}(T-,x)-\psi(x) \geq 0,
\]
where the second inequality follows from \eqref{eq:tfspsl}. Let us focus on the convergence of $x_m$ to $y$. The sequence $(x_{m})$ converges (up to choosing a sub-sequence) to some $z \in \overline{B(y,1)}$. 
By construction, $s_m \leq t_m$. Using this and  the choice of $(t_m,x_m)$ we obtain the following string of inequalities:
\[
\begin{split}
0 &\leq (v^-(T-,z)-\psi(z))-(v^{-}(T-,y)-\psi(y)) 
\\ &\leq  \liminf_{m \to \infty}\Big [(v^{-}(t_m,x_m)-\psi(x_m))-(v^{-}(s_m,y_m)-\psi(y_m))\Big ]
\\ &\leq \liminf_{m \to \infty}\Big [(v^{-}(t_m,x_m)-\psi_m(t_m,x_m))-(v^{-}(s_m,y_m)-\psi_m(s_m,y_m))
\\&- |x_m-y|^4+\frac{T-t_m}{(T-s_m)^2}  + |y_m-y|^4-\frac{T-s_m}{(T-s_m)^2}\Big ]
\\&\leq  \liminf_{m \to \infty}\Big[- |x_m-y|^4+|y_m-y|^4\Big]= -|z-y|^4,
\end{split}
\]
which proves that $z=y$.

We know that  $(t_m,x_m)$ is a minimizer of $v^--\psi_m$ over  $[s_m,T] \times \overline{B(y,\varepsilon)}$ by definition, and we also know that $s_m\leq t_m<T$ for large $m$.  Since $x_m\rightarrow y$, we conclude that (for $m$ large enough) we have
$(v^--\psi_m)(t_m, x_m)\leq (v^--\psi_m)(t,x)$ for $ t_m\leq t<T, \ |x-x_m|\leq \varepsilon/2.$
While this does not mean that $(t_m,x_m)$ is a local interior min for $v^--\psi_m$ (because we may have $t_m=s_m$), it does mean that we have a local ``parabolic interior minimum''. It is well known that, for example from \cite{CIL},  for parabolic equations, a ``parabolic interior minimum'' is enough to use $\psi_m$ as a test function at $(t_m,x_m)$, and, therefore first conclude that 
$$ - D_t \psi_m (t_m,x_m)-H(t_m,x_m,D_x \psi _m(t_0,x_0),D^2_x \psi _m (t_m,x_m)) \geq 0,
$$
so $H(t_m,x_m,D_x \psi _m(t_0,x_0),D^2_x \psi _m (t_m,x_m))<\infty$ and, consequently, 
\[
G(t_m,x_m,D_x \psi_m(t_m,x_m), D^2_x \psi_m(t_m,x_m)) \geq 0.
\]
Now the claim of this step follows from the continuity of $G$ and the fact that $x_m \to y$, as the derivatives of $\psi_m$ with respect to $x$ converge to those of $\psi$.


 
 \end{proof}
 \section{Stochastic super-solutions}\label{sec:ssuper}
 
In this section we consider the weak formulation of the stochastic control problem.
\begin{Assumption}\label{ass:weaksoln}
  We assume that  the coefficients $b :[0,T]\times \mathbb{R}^d \times U \rightarrow \mathbb{R}^d$ and $\sigma :[0,T]\times \mathbb{R}^d \times U\rightarrow \mathbb{M}_{d,d'}(\mathbb{R})$ are continuous. 
\end{Assumption}
\begin{Definition}
  For each $(s,x)$ we denote by $\mathcal{U}_{s,x}$ the set of weak admissible controls for the  \eqref{eq:SDE}, by which we mean a 
 $$\Big( \Omega ^{s,x}, \mathcal{F}^{s,x}, ( \mathcal{F}^{s,x}_t) _{s\leq t\leq T}, \mathbb{P}^{s,x}, (W^{s,x}_t)_{s\leq t\leq T},(X^{s,x}_t)_{s\leq t\leq T}, (u_{t})_{s\leq t\leq T}\bigg),$$
where \begin{enumerate}
\item $( \Omega ^{s,x}, \mathcal{F}^{s,x}, ( \mathcal{F}^{s,x}_t) _{s\leq t\leq T}, \mathbb{P}^{s,x})$ is an arbitrary stochastic basis satisfying the usual conditions, 
  \item $W^{s,x}$ is  a $d'$-dimensional Brownian motion with respect to the filtration $( \mathcal{F}^{s,x}_t) _{s\leq t\leq T}$,
\item  $u$ is a predictable and uniformly bounded $U$-valued process,
 \item  $X^{s,x}$ is a  continuous and adapted process satisfying \eqref{eq:SDE} with initial condition $X_s=x\in \mathcal{O}$, and
$\mathbb{P}^{s,x}(X^{s,x}_t\in \mathcal{O},\ s\leq t\leq T)=1$ together with
$$\mathbb{E}^{s,x}\left[\sup _{s\leq t\leq T}\psi (X^{s,x}_t)\right]<\infty,$$ 
\end{enumerate}
for the gauge function $\psi$ in Section~\ref{sec:subsoln}.
  \end{Definition}
   Now, for some measurable function $g:\mathcal{O}\rightarrow \mathbb{R}$, we denote by 
 \begin{equation}\label{value-weak}\mathfrak{V}(s,x):=\sup _{\mathcal{U}^{s,x}}\mathbb{E}^{s,x}[g(X^{s,x}_T)],\end{equation}
 the value function of the weak  control problem. 
 \begin{Assumption}\label{g}
 The pay-off function $g$ is  an upper semi-continuous function satisfying $|g(\cdot)|\leq C \psi(\cdot)$.
\end{Assumption}
 \begin{Remark}
\begin{enumerate}
\item Because of the growth assumption on weakly controlled solutions, 
$\mathbb{E}^{s,x}[g(X^{s,x}_T)]$ is well defined and finite, so $ \mathfrak{V}>-\infty.$
\item  When both are well-defined it clearly holds that $V \leq \mathfrak{V}$. 
\end{enumerate}
 \end{Remark}
 Our goal in this section is to construct an upper bound of $\mathfrak{V}$ that is a viscosity sub-solution.

\begin{Definition}\label{def:supersolution}
 The set of stochastic super-solutions for the parabolic PDE \eqref{eq:PDE}, denoted by $\mathcal{V}^+$, is  the set of functions $v:[0,T]\times \mathcal{O}^d\rightarrow \mathbb{R}$ which have the following properties:
\begin{enumerate}
\item They are continuous and satisfy the terminal condition $v(T,\cdot) \geq g(\cdot)$ together with the growth condition
$$|v(t,x)|\leq C(v) \psi (x), 0\leq t\leq T,\  x\in \mathcal{O}.$$

\item 
 For each $(s,x)\in[0,T]\times \mathcal{O}$, and each weak  control 
  $$\Big( \Omega ^{s,x}, \mathcal{F}^{s,x}, ( \mathcal{F}^{s,x}_t) _{s\leq t\leq T}, \mathbb{P}^{s,x}, (W^{s,x}_t)_{s\leq t\leq T},(X^{s,x}_t)_{s\leq t\leq T}, (u_{t})_{s\leq t\leq T}\bigg),$$
  the process $(u(t,X^{s,x}_t))_{s\leq t\leq T}$ is a super-martingale on $(\Omega ^{s,x}, \mathbb{P}^{s,x})$ with respect to the filtration 
$( \mathcal{F}^{s,x}_t) _{s\leq t\leq T}$.
\end{enumerate}
\end{Definition}
 
 \begin{Assumption}\label{ass:V+nepty}
 $\mathcal{V}^+ \neq \emptyset$.
 \end{Assumption}
 \begin{Remark}
 Assumption~\ref{ass:V+nepty} is satisfied, for example, when $g$ is bounded from above.
 \end{Remark}

 \begin{Theorem}\label{thm:mthm2}\emph{(The infimum of stochastic super-solutions is a viscosity sub-solution)}
Let Assumptions ~\ref{as:bnd}-(2),~\ref{ass:weaksoln}, \ref{g}, and \ref{ass:V+nepty} hold true.  Then $v^+=\inf _{v\in \mathcal{V}^+}v$ is a viscosity sub-solution of \eqref{eq:PDE}. Moreover, 
the USC function $v^{+}(T,\cdot)$ is a viscosity sub-solution of \eqref{eq:bndry}.
 \end{Theorem}
 \begin{proof}\hfill 
 \\ \textbf{Step 1.} The fact that $v^{+} \geq \mathfrak{V}$ follows directly from the definition of the class of stochastic sub-solutions and by the definition of $\mathcal{U}$. \\
\noindent \textbf{Step 2.} \emph{The interior sub-solution property.}
 Let $\varphi:[0,T] \times \mathcal{O} \to \mathbb{R}^d$ be a $C^{1,2}$-test function such that $v^{+}-\varphi$ attains a strict local maximum equal to zero at some parabolic interior point $(t_0,x_0) \in [0,T) \times \mathbb{R}^d$, where the viscosity sub-solution property fails, i.e.,
 \[
 \min\{-\varphi_t(t_0,x_0)-H(t,x,\varphi_x(t_0,x_0),\varphi_{xx}(t_0,x_0)), G(t,x,\varphi_x(t_0,x_0), \varphi_{xx}(t_0,x_0))\}>0.
 \]
Then since $G$ is continuous and $H$ is continuous in the interior of its domain it follows that there exists a small enough ball $B(t_0,x_0,\eps)$ such that, for all $(t,x) \in B(t_0,x_0,\eps)$ we have:
\[
\min\{-\varphi_t(t,x)-H(t,x,\varphi_x(t,x),\varphi_{xx}(t,x)), G(t,x,\varphi_x(t,x), \varphi_{xx}(t,x))\}>0.
 \]   
 Now the rest of the proof  of this step is very similar to the  corresponding step in the proof of Theorem 2.1 in \cite{bs2012a}, but much simplified by following the stopping idea in the proof of Theorem \ref{thm:mthm1} (step 2) above.
 For the sake of completeness and the convenience of the reader we actually include the remaining part of the  proof.
 The function $v^{+}-\varphi$ is uper semi-continuous and $\overline{B(t_0,x_0,\eps)}-B(t_0,x_0,\eps/2)$ is compact, there exists a $\delta>0$ satisfying
\[
v^{+}+\delta \leq \varphi \quad \text{on}\;\overline{B(t_0,x_0,\eps)}-B(t_0,x_0,\eps/2).
\]
Using Proposition 4.1 in \cite{bs2012a} together with  the obvious observation that the minimum of two stochastic super-solutions is also a stochastic super-solution, we obtain a (countable) decreasing sequence of stochastic super-solutions
$v_n\searrow v^+$. Now, since $\varphi$ is continuous, as well as $v_n$'s, we can use  once again a Dini argument (identical to the one  in Lemma 2.4 of \cite{bs2012b})  to conclude  that for $\delta' \in (0,\delta)$ there exists a stochastic super-solution $v=v_n$ (for some large enough $n$) such that
\[v+\delta' \leq \varphi \quad \text{on}\;\overline{B(t_0,x_0,\eps)}-B(t_0,x_0,\eps/2).
\]
Choosing $\eta \in (0,\delta')$ small enough we have that the function
\[
\varphi^{\eta}:=\varphi-\eta
\]
satisfies
\[
\begin{split}
-\varphi^{\eta}_t (t,x)- H(t,x,\varphi ^{\eta}_x(t,x),\varphi^{\eta}_{xx}(t,x))>0,  (t,x) \in B(t_0,x_0,\eps),
\\ \varphi^{\eta}(t,x)>v(t,x),\ (t,x) \in \overline{B(t_0,x_0,\eps)}-B(t_0,x_0,\eps/2),
\end{split}
\]
and
\[
\varphi^{\eta}(t_0,x_0)=v^{+}(t_0,x_0)-\eta <v^+(t_0,x_0).
\]
Now we define, similarly to Step 2 above,  
$$v^{\eta}=
\left \{
\begin{array}{l}
 v\wedge  \varphi ^{\eta} \textrm{~on~} \overline{ B(t_0, x_0, \varepsilon)},\\
v \textrm{~outside~}\overline{ B(t_0, x_0, \varepsilon)}.
\end{array}
\right.
$$
Clearly, $v^{\eta}$ is continuous and $v^{\eta}(t_0,x_0)=\varphi^{\eta}(t_0,x_0)>v^{-}(t_0,x_0)$. And since $\eps$ can be chosen so that $T>t_0+\eps$, $v^{\eta}$ satisfies the terminal condition.  Again, the growth condition in (i)  Definition~\ref{defn:stsubsoln} holds for $v^{\eta}$, since such growth condition holds for the approximate infimum $v$.
We now only need to show  that $v^{\eta}$ satisfies (ii) in Definition~\ref{def:supersolution}  to get a contradiction and complete the proof. Fix an admissible weak control 
$$\Big( \Omega ^{s,x}, \mathcal{F}^{s,x}, ( \mathcal{F}^{s,x}_t) _{s\leq t\leq T}, \mathbb{P}^{s,x}, (W^{s,x}_t)_{s\leq t\leq T},(X^{s,x}_t)_{s\leq t\leq T}, (u_{t})_{s\leq t\leq T}\bigg).$$
Fix now $s\leq \tau \leq \rho\leq T$ two stopping times of the filtration   $( \mathcal{F}^{s,x}_t) _{s\leq t\leq T}$. Denote, similarly to Step 2,  by $A$ the event
$$A=\{(\tau, X^{s,x}_{\tau})\in B(t_0, x_0, \varepsilon/2) \textrm{~and~} \varphi ^{\eta}(\tau , X^{s,x}_{\tau})<v(\tau ,X^{s,x}_{\tau})\}.$$
Denote by  $\tau _1$ the first time after $\tau$ when the diffusion hits the boundary of $B(t_0, x_0, \varepsilon/2)$:
$$\tau _1=\inf \{\tau\leq  t\leq T| X^{s,x}_t\in \partial B(t_0, x_0, \varepsilon/2)\}.$$
  On the event $A$, $\varphi ^{\eta}(\cdot, X^{s,x}_{\cdot})$ is a continuous super-martingale up to $\rho \wedge \tau _1$ (because of It\^o's formula together with the fact that $\varphi ^{\eta}$ is bounded in the interior ball), which reads
$$1_A \varphi ^{\eta}(\tau, X^{s,x}_{\tau})\geq \mathbb{E}^{s,x}[1_A \varphi ^{\eta}(\rho \wedge \tau _1, X_{\rho \wedge \tau _1}^{s,x})|\mathcal{F}^{s,x}_{\tau}] \ \mathbb{P}^{s,x}-a.s.$$
Since
$1_A \varphi ^{\eta}(\rho \wedge \tau _1, X_{\rho \wedge \tau _1}^{s,x}) \geq 1_A v ^{\eta}(\rho \wedge \tau _1, X_{\rho \wedge \tau _1}^{s,x}),$
we have 
$$1_A v^{\eta}(\tau, X^{s,x}_{\tau})=1_A \varphi ^{\eta}(\tau, X^{s,x}_{\tau})\geq \mathbb{E}^{s,x}[1_A v ^{\eta}(\rho \wedge \tau _1, X_{\rho \wedge \tau _1}^{s,x})|\mathcal{F}^{s,x}_{\tau}]  \ \mathbb{P}^{s,x}-a.s.$$
Next, we use the optional sampling theorem applied to the continuous super-martingale $v(\cdot, X^{s,x}_{\cdot})$ in between the stopping times $\tau\leq \rho \wedge \tau _1$,  together with the observation that  $v=v^{\eta}$ everywhere outside the open ball $ B(t_0,x_0,\eps/2)$, to obtain:
\begin{eqnarray*}
1_{A^c} v^{\eta}(\tau, X^{s,x}_{\tau})=1_{A^c} v (\tau, X^{s,x}_{\tau})\geq \mathbb{E}^{s,x}[1_{A^c} v (\rho \wedge \tau _1, X_{\rho \wedge \tau _1}^{s,x})|\mathcal{F}^{s,x}_{\tau}]\\
\geq \mathbb{E}^{s,x}[1_{A^c} v ^{\eta}(\rho \wedge \tau _1, X_{\rho \wedge \tau _1}^{s,x})|\mathcal{F}^{s,x}_{\tau}], \ \mathbb{P}^{s,x}-a.s.\end{eqnarray*}
Putting the above together, we obtain:
 \begin{equation}\label{eq:supermp}
 v^{\eta}(\tau,X^{s,x}_{\tau}) \geq \E^{s,x}\left[v^{\eta}\left(\rho \wedge \tau_1, X^{s,x}_{\rho \wedge \tau_1}\right)\bigg|\mathcal{F}^{s,x}_{\tau} \right]\ \mathbb{P}^{s,x}- a.s.
 \end{equation}
Let us again introduce the  notation: $B=\{\rho \leq \tau _1\}\in \mathcal{F}^{s,x}_{\tau _1 \wedge \rho}.$ We know that, on the boundary $ \partial B(t_0, x_0, \varepsilon/2)$, $v=v^\eta.$  Together with the optional sampling theorem applied to the continuous super-martingale $v(\cdot, X^{s,x}_{\cdot})$ between $\tau _1\wedge \rho$ and $\rho$ we have
\begin{eqnarray*}
1_{B^c} v^{\eta}(\tau _1, X^{s,x}_{\tau _1})=
1_{B^c} v (\tau _1, X^{s,x}_{\tau _1})
\geq \mathbb{E}^{s,x}[1_{B^c} v (\rho , X_{\rho}^{s,x})|\mathcal{F}^{s,x}_{\tau _1}]\\
\geq \mathbb{E}[1_{B^c} v ^{\eta}(\rho, X_{\rho}^{s,x})|\mathcal{F}^{s,x}_{\tau _1}],\ \mathbb{P}^{s,x}-a.s.
\end{eqnarray*}
We finally rewrite the RHS in 
\eqref{eq:supermp} as
$$\E^{s,x}\left[v^{\eta}\left(\rho \wedge \tau_1, X^{s,x}_{\rho \wedge \tau_1}\right)\bigg|\mathcal{F}^{s,x}_{\tau} \right]=\E^{s,x}\left[ 1_B v^{\eta}\left(\rho, X^{s,x}_{\rho}\right) 
+ 1_{B^c} v^{\eta}(\tau _1, X^{s,x}_{\tau_1})\bigg|\mathcal{F}^{s,x}_{\tau} \right],
$$
and use the tower property to obtain \begin{equation*}
 v^{\eta}(\tau,X^{s,x}_{\tau}) \geq \E ^{s,x}\left[v^{\eta}\left(\rho, X^{s,x}_{\rho}\right)\bigg|\mathcal{F}^{s,x}_{\tau} \right] \  \mathbb{P}^{s,x}-a.s.
 \end{equation*}
  Since this happens for any stopping times $s\leq \tau \leq\rho \leq T$ of the filtration  $( \mathcal{F}^{s,x}_t) _{s\leq t\leq T}$, we have, indeed, that $v^{\eta}$ is a stochastic super-solution, leading to a contradiction and completing the proof.

 \noindent \textbf{Step 3.} \emph{The boundary condition.}
 
Let $x_0 \in \mathcal{O}$ and $\psi$ be a smooth function on $\mathcal{O}$ such that 
\[
0=v^+(T,x_0)-\psi(x_0)=\max_{\mathcal{O}}(v^+(T,x)-\psi(x)).
\] 
 Assume, in addition, without losing generality, that the maximum is strict.
Let us assume, by contradiction, that 
\begin{equation}\label{eq:cntrpstv}
G(T,x_0,\psi_x(x_0), \psi_{xx}(x_0)) >0 \ \textrm{~and~}v^+(T,x_0)>g(x_0).
\end{equation}
Since $G$ is continuous, and, in addition, $G$ is finite {\it and} continuous in the open  set $G>0$, we conclude that, there exists  small $\varepsilon, \delta _0  >0$ and a finite constant $C$ such that
$$H(t, x, \psi _x (x), \psi _{xx}(x))< C,\ \ T-t\leq \delta _0 , |x-x_0|\leq \varepsilon.$$
In addition, we also have (for small enough $\varepsilon$)
$$\psi (x)\geq g(x)+\varepsilon,\ \ |x-x_0|\leq \varepsilon.$$
Now, the whole idea is based on constructing a local super-solution 
$$\psi^k(t,x)=\psi (x)+k (T-t)$$
for large $k$, {\it by decoupling the bounds $\delta$ and $\varepsilon$ in the estimate above}, then pushing it slightly down. Namely, we will make $\delta $ much smaller than $\varepsilon$. Fix $\delta _0$ and $\varepsilon$ as above. Denote by 
$$h(\delta )=\sup _{T-t\leq \delta , \frac{\varepsilon}2\leq |x-x_0|\leq \varepsilon}\Big (v^+(t,x)-\psi (x)\Big),\ 0<\delta <\delta _0.$$
Interpreting $\psi$ as a continuous function of two variables $(t,x)$, which actually does not depend on $t$ and taking into account that $v^+$ is USC, there exist a point where the maximum above is attained, i.e.
$$h(\delta)=v^+(t_{\delta}, x_{\delta})-\psi (x_{\delta}).$$
By compactness, we can subtract a sub-sequence (we still denote it as $\delta \searrow 0$) such that
$$(t_{\delta}, x_{\delta})\rightarrow (T, x^*),\ \ \frac{\varepsilon}2\leq |x^*-x_0|\leq \varepsilon.$$
Since $v^+$ is USC, we conclude that
\begin{equation}
\begin{split}
\limsup_{\delta \searrow 0} h(\delta)&=\limsup _{\delta \searrow 0} \Big (v^+(t_{\delta}, x_{\delta})-\psi (x_{\delta})\Big)\\
& \leq 
v^+(T, x^*)-\psi (x^*)\leq \sup _{\frac{\varepsilon}2\leq |x-x_0|\leq \varepsilon}\Big (v^+(T,x)-\psi (x)\Big )<0,
\end{split}
\end{equation}
where the last inequality follows from the fact that we have a strict max at $x_0$ and the last supremum is actually attained. Therefore, we can choose $\delta<\delta_0$ small enough such that $h(\delta)<0$.
Now, for {\it this} fixed $\delta$, with the notation
$$\delta '=-h(\delta)>0$$
we have 
\begin{equation}\label{eq:donut}
v^+(t,x)\leq \psi (x)-\delta' ,\ \  T-t\leq \delta , \frac{\varepsilon}2\leq |x-x_0|\leq \varepsilon.
\end{equation}
 Denote by $D$ the compact ``rectangular donut'' 
$$D=\{(t,x)|T-t\leq \delta , |x-x_0|\leq \varepsilon\}-
\{(t,x)|T-t<\delta/2 , |x-x_0|<\varepsilon/2 \}.$$
 Since, by USC, $v^+$ is bounded on $\{\delta /2\leq T-t\leq \delta , |x-x_0|\leq \varepsilon/2\}$ we  can choose $k$ large enough such that 
$$v^+\leq \psi ^k-\delta ' \textrm{~on~} \{\delta /2\leq T-t\leq \delta , |x-x_0|\leq \varepsilon/2\}.$$
Together with \eqref{eq:donut}, we obtain
$$v^+\leq \psi ^k-\delta ' \textrm{~on~}  D.$$
In addition
 \[
 H(t,x,\psi^k _x(t,x),\psi^k _{xx}(t,x))= H(t,x,\psi _x(t,x),\psi _{xx}(t,x)) \leq C,  \quad T-t\leq \delta, |x-x_0|\leq \varepsilon,
 \]
so 
$$
 -\psi^k_t(t,x)- H(t,x,\psi^k_x(t,x),\psi^k_{xx}(t,x)) \geq k- C>0, \;\; \\
$$
for $k$ even larger, if $T-t\leq \delta, |x-x_0|\leq \varepsilon$.
Following the proof of Step 2 in Theorem \ref{thm:mthm1}, we use again Proposition 4.1 in \cite{bs2012a} and  the Dini argument to obtain  a stochastic sub-solution $v \in \mathcal{V}^+$ such that
$v\leq \psi ^k-\delta '/2\ \ \ \textrm{~on~}D.$

Now let $\eta<\delta'/2<\eps$ and define
$$v^k=
\left \{
\begin{array}{ll}
 v\wedge \Big ( \psi ^k -\eta  \Big ),\ \ T-t\leq \delta, |x-x_0|\leq \varepsilon, \\
v, \textrm{~otherwise}.
\end{array}
\right.$$
It follows, using the same stopping argument as in the proof of Theorem \ref{thm:mthm1}, that $v^k \in \mathcal{V}^+$. But we also have that $v^k(T,x_0)=v^+(T,x_0)-\eta <v^+(T,x_0)$, which contradicts the definition of the function $v^+$.

\end{proof}

\section{Verification by comparison}\label{sec:verification}
 Before we go ahead, we recall that our analysis rests on the assumption of the existence of stochastic sub and super-solutions. Such assumption may actually be non-trivial to check, especially given the choice of the gauge function $\psi$ (see Remark \ref{face-lift} below).
\begin{Assumption}\label{ass:elptccmp}
There is a comparison principle between USC sub-solutions and LSC super-solutions within the class
$|w|\leq C\psi$
for the PDE
\begin{equation}\label{eq:bndrypde}
\min[w(x)-g(x),G(T,x,w_x(x), w_{xx}(x))]=0, \quad \text{on} \quad \mathcal{O}.
\end{equation}
\end{Assumption}
\begin{Remark}\label{face-lift}
The choice of $\psi$ can make a difference whether we have or not  a comparison result for \eqref{eq:bndrypde}. As mentioned, we do not have boundary conditions per-se (this carries over to \eqref{eq:bndrypde}), but the  information on behavior of solutions near the boundary might, sometimes, be contained in the choice of $\psi$. Therefore, if one wants, for example, to add a constant to $\psi$, having an easier time checking for the existence of stochastic super-solutions or sub-solutions, uniqueness may be lost in \eqref{eq:bndrypde}.

\end{Remark}

\begin{Lemma}\label{lem:bndrycnd}
Let us suppose that Assumption~\ref{ass:elptccmp}
and assumptions in both Theorem~\ref{thm:mthm1} and Theorem~\ref{thm:mthm2} hold.  Then:

\begin{equation}\label{eq:llql}
v^{-}(T-,\cdot)=v^{+}(T,\cdot)=\hat{g}(\cdot),
\end{equation}
where $\hat{g}$ is the unique continuous viscosity solution of \eqref{eq:bndrypde}. In addition, both the strong and the weak value functions have well defined limits at $T$, equal to the terminal condition $\hat{g}$:
$$\lim _{(t<T, x')\rightarrow (T,x)} V(t,x')=\lim_{(t<T, x')\rightarrow (T,x)}\mathfrak{V}(t,x')=\hat{g}(x),\ x\in \mathcal{O}.$$
\end{Lemma}
\begin{proof}

It follows from their definitions that $v^- \leq v^+$. Since $v^+$ is USC, then
\begin{equation}
\label{eq:limits}
v^-(T-,x)=\liminf _{(t<T, x')\rightarrow (T,x)} v^-(t,x')\leq \limsup _{(t<T, x')\rightarrow (T,x)}v^+(t,x')\leq v^+(T,x).
\end{equation}
Moreover, $v^{-}(T-,\cdot)$ is a LSC viscosity super-solution of \eqref{eq:bndrypde} as a result of Theorem~\ref{thm:mthm1}, and $v^{+}(T, \cdot)$ is an USC viscosity sub-solution of the same PDE due to Theorem~\ref{thm:mthm2}. 
In addition, under the assumptions that both $\mathcal{V}^-$ and $\mathcal{V}^+$ are non-empty, we have the bounds
$$|v^-|, | v^+|\leq C\psi,$$
obtaining therefore similar growth conditions for $v^+(T,\cdot)$ and $v^-(T-,\cdot)$.
Thanks to the comparison assumption, it follows that $v^{+}(T, \cdot)=v^{-}(T-,\cdot)$ and the common value is the unique continuous viscosity solution of \eqref{eq:bndrypde} that we denote by $\hat{g}$.

In order to prove the second statement, we only need to note that
$$v^-\leq V\leq \mathfrak{V}\leq v^+ $$
and plug the equality $v^-(T-, \cdot)=v^+(T,\dot)=\hat{g}(\cdot)$ in \eqref{eq:limits}.

\end{proof}
\begin{Proposition}\emph{($G$ upper envelope of $g$.)}
Under Assumption \ref{ass:elptccmp}, the function $\hat{g}$
is  the smallest (continuous) function above $g$ which is a viscosity super-solution of 
\begin{equation}\label{eq:layer}
G(T,x,w_x(x), w_{xx}(x))=0, \quad \text{on}\;\; \mathcal{O}.
\end{equation}
\end{Proposition}
\begin{proof} We know that $\hat{g}\geq g$ and that $\hat{g}$ is a viscosity super-solution of \eqref{eq:layer}. Consider now a $w\geq g$ and $w$ is a super-solution of \eqref{eq:layer}. Then, $w$ is a super-solution of \eqref{eq:bndrypde}. Since $\hat{g}$ is a solution of \eqref{eq:bndrypde} and we have a comparison result, then $\hat{g}\leq w$.

\end{proof}
\begin{Remark}
When the space of controls is compact, one may take $G$ to be equal to a positive constant. In that case $g=\hat{g}$.
\end{Remark}

\begin{Definition}
We say that a comparison principle for \eqref{eq:PDE} holds if, whenever we have an upper semi-continuous viscosity sub-solution $v$, and a lower semi-continuous viscosity super-solution $w$ satisfying growth conditions
$|v|, |w|\leq C(1+\psi)$
 with $v(T,\cdot) \leq w(T,\cdot)$ on $\mathcal{O}$, then $v \leq w$.
\end{Definition}

\begin{Remark} One cannot expect comparison up to time $t=0$ for semi-continuous viscosity semi-solutions, unless the viscosity property holds in the whole parabolic interior, which includes $t=0$. This can be seen, for example, from \cite{CIL} and \cite{dfo}. The reader may note that we did prove the viscosity semi-solution property for $v^-$ and $v^+$ in the parabolic interior.
\end{Remark}

Now we are ready to state the main result of this section, which follows as a corollary of Theorems~\ref{thm:mthm1} and \ref{thm:mthm2} and Lemma~\ref{lem:bndrycnd}.
\begin{Theorem}\label{thm:main-comp}
Let us assume that a comparison principle for \eqref{eq:PDE} holds. Moreover, we assume that Assumption~\ref{ass:elptccmp}
and assumptions in both Theorem~\ref{thm:mthm1} and Theorem~\ref{thm:mthm2} hold. Then, there exists a unique continuous (up to $T$) viscosity solution $v\in C([0,T]\times \mathcal{O})$ of the PDE \eqref{eq:PDE} with terminal condition $v(T,\cdot)=\hat{g}(\cdot)$, satisfying the growth condition $|v|\leq C \psi$.
Before time $T$ we have:

\[
v(t,x)=v^{-}(t,x)=v^{+}(t,x)=V(t,x) = \mathfrak{V}(t,x) \quad (t,x) \in [0,T) \times \mathcal{O}.
\]
\end{Theorem}
\begin{proof}
Since Assumption~\ref{ass:elptccmp} holds, then 
$v^-(T-, \cdot)=v^+(T, \cdot)=\hat{g}(x).$
We now define the (still LSC) function 
$$w(t,x)=
\left \{
\begin{array}{ll}
v^-(t,x),\ \ 0\leq t<T, x\in \mathcal{O}\\
\hat{g}(x),\ \ t=T, x\in \mathcal{O}.
\end{array}\right.$$
By definition, $w\leq v^+$. At the same time,
the function $w$ is  a LSC viscosity super-solution and $v^+$ is a  USC viscosity sub-solution of \eqref{eq:PDE}. Since
$v^+(T \cdot)=w(T, \cdot)$ 
 we can use comparison to conclude that $v^+\leq w$, so
$$v^+=w\in  C([0,T]\times \mathcal{O}).$$
Denoting by 
$v=w=v^+$, the proof is complete.
\end{proof}
\begin{Remark}
When the controls are unbounded, the value function may display a discontinuity at the terminal time $T$, as  we expect that $v(T-, \cdot)=\hat{g}$ and  $v(T,\cdot)=g$. (If $t \neq T$, it follows from the above theorem that the value function is continuous.) The discontinuity was already observed by Krylov in \cite{MR2723141} on page 252, but the question of what the correct boundary condition should be was left open. 
For a particular model  of super-hedging, an answer was given in \cite{bcs}.  The technical arguments  to treat such behavior close to the final time horizon were extended to more general models of super-hedging in \cite{cpt}, \cite{st}. A summary of such arguments can also be found in  \cite{st-lnm} or in the textbook \cite{MR2533355}.
One of our contributions is to show that this boundary condition holds without relying on the DPP. The proof of the boundary condition comes out as a simple conclusion from the Stochastic Perron method.
\end{Remark}
\begin{Remark}\label{rem:fleming-vermes} (Fleming-Vermes)
As we have mentioned in the introduction, using our notation, Fleming-Vermes \cite{fleming-vermes} and \cite{fleming-vermes-2} prove that (with the notation \eqref{value-weak}) we have 
$$V=\mathfrak{v}=\inf \{\textrm{classic super-solutions}\},$$
under some technical assumptions (in particular, there is no boundary layer).
The proof uses a sophisticated approximation/separation argument, and the probabilistic representation of $V$, $\mathfrak{v}$. 




\end{Remark}
 The program we propose in the present paper  can be summarized as 

\vspace{.1in}

\noindent \fbox{Theorem \ref{thm:mthm1}+  Theorem \ref{thm:mthm2}+Comparison $\rightarrow V=\mathfrak{v}=$ unique viscosity solution}

\vspace{.1in}

However, in the absence of a comparison result for semi-continuous viscosity solutions, little can actually be said about the properties of the value function, following this approach.

\bibliographystyle{amsplain}


\begin{thebibliography}{10}

\bibitem{bs2012b}
E.~Bayraktar and M.~S\^{i}rbu, \emph{Stochastic {P}erron's method and
  verification without smoothness using viscosity comparison: obstacle problems
  and dynkin games}, to appear in the Proceedings of the American Mathematical
  Society (2012), http://arxiv.org/abs/1112.4904.

\bibitem{bs2012a}
\bysame, \emph{Stochastic {P}erron's method and verification without smoothness
  using viscosity comparison: the linear case}, Proceedings of the American
  Mathematical Society \textbf{140} (2012), 3645--3654.

\bibitem{wdpp}
Bruno Bouchard and Nizar Touzi, \emph{Weak dynamic programming principle for
  viscosity solutions}, SIAM J. Control Optim. \textbf{49} (2011), no.~3,
  948--962.

\bibitem{bcs}
M.~Broadie, J.~Cvitani{\'c}, and M.~Soner, \emph{Optimal replication of
  contingent claims under portolio constraints}, Review of Financial Studies
  \textbf{11} (1998), 59--79.

\bibitem{claisse-talay-tan}
J.~Claisse, D.~Talay, and X.~Tan, \emph{A note on solutions to controlled
  martingale problems and their conditioning}, preprint available at
  http://www.cmapx.polytechnique.fr/$\sim$tan/JDX.pdf, April 2013.

\bibitem{CIL}
M.~Crandall, H.~Ishii, and P.-L. Lions, \emph{User's guide to viscosity
  solutions of second-order partial differential equations}, Bull. Amer. Math.
  Soc \textbf{27} (1992), 1--67.

\bibitem{cpt}
J.~Cvitani{\'c}, H.~Pham, and N.~Touzi, \emph{Super-replication in stochastic
  volatility models under portfolio constraints}, Journal of Applied
  Probability \textbf{36} (1999), no.~2, 523--545.

\bibitem{dfo}
J.~Diehl, P.~K. Friz, and H.~Oberhauser, \emph{Parabolic comparison revisited
  and applications}, preprint, http://arxiv.org/abs/1102.5774, 2011.

\bibitem{tz1}
I.~Ekren, C.~Keller, N.~Touzi, and J.~Zhang, \emph{On viscosity solutions of
  path dependent pdes}, to appear in the Annals of Probability.

\bibitem{tz2}
I.~Ekren, N.~Touzi, and J.~Zhang, \emph{Viscosity solutions of fully nonlinear
  parabolic path dependent pdes: Part i},  (2012), preprint,
  arXiv:1210.0006[math.PR].

\bibitem{tz3}
\bysame, \emph{Viscosity solutions of fully nonlinear parabolic path dependent
  pdes: Part ii},  (2012), preprint, arXiv:1210.0007[math.PR].

\bibitem{fleming-vermes-2}
H.~W. Fleming and D.~Vermes, \emph{Generalized solutions in the optimal control
  of diffusions}, Stochastic differential systems, stochastic control theory
  and applications (W.~H. Fleming and P.L. Lions, eds.), pp.~119--127.

\bibitem{fleming-vermes}
W.~H. Fleming and D.~Vermes, \emph{Convex duality approach to the optimal
  control of diffusions}, SIAM Journal on Control and Optimization \textbf{27}
  (1989), no.~5, 1136--1155.

\bibitem{ishii}
H.~Ishii, \emph{Perron's method for {Hamilton-Jacobi} equations}, Duke
  Mathematical Journal \textbf{55} (1987), no.~2, 369--384.

\bibitem{MR2723141}
N.~V. Krylov, \emph{Controlled diffusion processes}, Stochastic Modelling and
  Applied Probability, vol.~14, Springer-Verlag, Berlin, 2009, Translated from
  the 1977 Russian original by A. B. Aries, Reprint of the 1980 edition.

\bibitem{MR2533355}
H.~Pham, \emph{Continuous-time stochastic control and optimization with
  financial applications}, Stochastic Modelling and Applied Probability,
  vol.~61, Springer-Verlag, Berlin, 2009.

\bibitem{st-lnm}
H.~Mete Soner and Nizar Touzi, \emph{The problem of super-replication under
  constraints}, Paris-{P}rinceton {L}ectures on {M}athematical {F}inance, 2002,
  Lecture Notes in Math., vol. 1814, Springer, Berlin, 2003, pp.~133--172.

\bibitem{st}
M.~Soner and N.~Touzi, \emph{Stochastic target problems, dynamic programming,
  and viscosity solutions}, SIAM Journal on Control and Optimization
  \textbf{41} (2002), no.~2, 404--424.


\bibitem{swiech-1}
A.~{\'S}wi{\c{e}}ch, \emph{Sub- and superoptimality principles of dynamic programming
  revisited}, Nonlinear Anal. \textbf{26} (1996), no.~8, 1429--1436.

\bibitem{swiech-2}
\bysame, \emph{Another approach to the existence of value functions
  of stochastic differential games}, J. Math. Anal. Appl. \textbf{204} (1996),
  no.~3, 884--897.



\bibitem{xyz}
X.~Zhou, J.~Yong, and X.~Li, \emph{Stochastic verification theorems within the
  framework of viscosity solutions}, SIAM Journal on Control and Optimization
  \textbf{35} (1997), 243--253.

\end{thebibliography}
\def\cprime{$'$}
\providecommand{\bysame}{\leavevmode\hbox to3em{\hrulefill}\thinspace}
\providecommand{\MR}{\relax\ifhmode\unskip\space\fi MR }
\providecommand{\MRhref}[2]{%
  \href{http://www.ams.org/mathscinet-getitem?mr=#1}{#2}
}
\providecommand{\href}[2]{#2}

 \end{document}